\documentclass[12pt,a4paper]{amsart}
\usepackage[utf8]{inputenc}
\usepackage[english]{babel}
\usepackage{graphicx}
\usepackage{amsmath}
\usepackage{amssymb}
\usepackage{amscd}
\usepackage{amsthm}
\usepackage{amscd}

\usepackage[matrix,arrow,curve]{xy}
\usepackage{stmaryrd}
\usepackage{hyperref}

\usepackage{tikz}
\usetikzlibrary{arrows.meta}
\usetikzlibrary{bending}

\usepackage{graphicx}
\usepackage{xspace}
\usepackage{enumerate}
\usepackage{verbatim}
\usepackage{geometry}
\usepackage{tikz-cd}

\DeclareMathOperator{\Spec}{\mathrm{Spec}}

\DeclareMathOperator{\rnk}{\mathrm{rk}}
\DeclareMathOperator{\Br}{\mathrm{Br}}

\address{Andrei Lavrenov, Mathematisches Institut der Universit\"at M\"unchen, The\-re\-sien\-str. 39, 80333 M\"unchen, Germany}
\email{avlavrenov@gmail.com}

\address{Victor Petrov,  Laboratory of Modern Algebra and Applications, St. Petersburg State University, 14th Line V.O. 29b, 199178 St. Petersburg, Russia 
and 
PDMI RAS, Fontanka emb. 27, 191023 St. Petersburg, Russia}
\email{victorapetrov@googlemail.com}



\begin{document}

\newcommand{\Sm}{\mathcal{S}^{\!}\mathsf{m}_k}
\newcommand{\Rings}{\mathcal{R}^{\!}\mathsf{ings}^*}
\newcommand{\rings}{\mathcal{R}^{\!}\mathsf{ings}}
\newcommand{\SmE}[1]{\mathcal{S}^{\!}\mathsf{m}_{#1}}
\newcommand{\Mot}[1]{\mathcal M^{\!}\mathsf{ot}_{#1}}
\newcommand{\Corr}[1]{\mathcal C^{\!}\mathsf{orr}_{#1}}
\newcommand{\MotF}[1]{\mathcal M_{#1}}
\newcommand{\QG}{\mathbb Q\Gamma}
\newcommand{\Inv}{\mathrm{Inv}(\QG)}
\newcommand{\Gal}{\mathrm{Gal}(k^{\mathrm{sep}}/k)}
\newcommand{\End}{\mathrm{End}}
\newcommand{\M}[1]{\mathcal{M}_{#1}}
\newcommand{\EG}{\!\,_EG}
\newcommand{\EGP}{\!\,_E(G/P)}
\newcommand{\EX}{\!\,_EX}
\newcommand{\XG}{\!\,_{\xi}G}
\newcommand{\XGP}{\!\,_{\xi}(G/P)}
\newcommand{\KQ}[1]{\mathrm K(n)^*\big(#1;\,\mathbb Q[v_n^{\pm1}]\big)}
\newcommand{\KZ}[1]{\mathrm K(n)^*\big(#1;\,\mathbb Z_{(p)}[v_n^{\pm1}]\big)}
\newcommand{\KZp}[1]{\mathrm K(n)^*\big(#1;\,\mathbb Z_p[v_n^{\pm1}]\big)}
\newcommand{\KF}[1]{\mathrm K(n)^*\big(#1;\,\mathbb F_p[v_n^{\pm1}]\big)}
\newcommand{\CHQ}[1]{\mathrm{CH}^*\big(#1;\,\mathbb Q[v_n^{\pm1}]\big)}
\newcommand{\KXZ}[1]{\!\,^{\mathrm K(n)\!}#1_{\,\mathbb Z_{(p)}[v_n^{\pm1}]}}
\newcommand{\KMotQ}{\Mot{\,\mathrm K(n)}}
\newcommand{\CHMotQv}{\Mot{\,\mathrm{CH}}}
\newcommand{\KXQ}[1]{\mathcal M_{\mathrm K(n)}(#1)}
\newcommand{\KMQ}{\mathcal M_{\,\mathrm K(n)}}
\newcommand{\CHMQv}{\MotF{\,\mathrm{CH}}}
\newcommand{\KCorrQ}{\Corr{\,\mathrm K(n)}}
\newcommand{\CHCorrQv}{\Corr{\,\mathrm{CH}}}
\newcommand{\CHCorrQ}{\Corr{\,\mathrm{CH}}}
\newcommand{\AMot}{\Mot A}

\newcommand{\e}{\varepsilon}
\newcommand{\con}{\ensuremath{\triangledown}}
\newcommand{\ra}{\ensuremath{\rightarrow}}
\newcommand{\tp}{\ensuremath{\otimes}}
\newcommand{\pr}{\ensuremath{\partial}}
\newcommand{\trigd}{\ensuremath{\triangledown}}
\newcommand{\dAB}{\ensuremath{\Omega_{A/B}}}
\newcommand{\QQ}{\ensuremath{\mathbb{Q}}\xspace}
\newcommand{\CC}{\ensuremath{\mathbb{C}}\xspace}
\newcommand{\RR}{\ensuremath{\mathbb{R}}\xspace}
\newcommand{\ZZ}{\ensuremath{\mathbb{Z}}\xspace}
\newcommand{\Zp}{\ensuremath{\mathbb{Z}_{(p)}}\xspace}
\newcommand{\Z}[1]{\ensuremath{\mathbb{Z}_{(#1)}}\xspace}
\newcommand{\NN}{\ensuremath{\mathbb{N}}\xspace}
\newcommand{\LL}{\ensuremath{\mathbb{L}}\xspace}
\newcommand{\inN}{\ensuremath{\in\mathbb{N}}\xspace}
\newcommand{\inQ}{\ensuremath{\in\mathbb{Q}}\xspace}
\newcommand{\inR}{\ensuremath{\in\mathbb{R}}\xspace}
\newcommand{\inC}{\ensuremath{\in\mathbb{C}}\xspace}
\newcommand{\OO}{\ensuremath{\mathcal{O}}\xspace}
\newcommand{\rarr}{\rightarrow}
\newcommand{\Rarr}{\Rightarrow}
\newcommand{\xrarr}[1]{\xrightarrow{#1}}
\newcommand{\larr}{\leftarrow}
\newcommand{\lrarr}{\leftrightarrows}
\newcommand{\rlarr}{\rightleftarrows}
\newcommand{\rrarr}{\rightrightarrows}
\newcommand{\al}{\alpha}
\newcommand{\bt}{\beta}
\newcommand{\ld}{\lambda}
\newcommand{\om}{\omega}
\newcommand{\Kd}[1]{\ensuremath{\Omega^{#1}}}
\newcommand{\KKd}{\ensuremath{\Omega^2}}
\newcommand{\vd}{\partial}
\newcommand{\PC}{\ensuremath{\mathbb{P}_1(\mathbb{C})}}
\newcommand{\PPC}{\ensuremath{\mathbb{P}_2(\mathbb{C})}}
\newcommand{\derz}{\ensuremath{\frac{\partial}{\partial z}}}
\newcommand{\derw}{\ensuremath{\frac{\partial}{\partial w}}}
\newcommand{\mb}[1]{\ensuremath{\mathbb{#1}}}
\newcommand{\mf}[1]{\ensuremath{\mathfrak{#1}}}
\newcommand{\mc}[1]{\ensuremath{\mathcal{#1}}}
\newcommand{\id}{\ensuremath{\mbox{id}}}
\newcommand{\dd}{\ensuremath{\delta}}
\newcommand{\bu}{\bullet}
\newcommand{\ot}{\otimes}
\newcommand{\boxt}{\boxtimes}
\newcommand{\op}{\oplus}
\newcommand{\mt}{\times}
\newcommand{\Gm}{\mathbb{G}_m}
\newcommand{\Ext}{\ensuremath{\mathrm{Ext}}}
\newcommand{\Tor}{\ensuremath{\mathrm{Tor}}}

\newcommand{\kn}[1]{\mathrm K(n)^*(#1)}
\newcommand{\ckn}[1]{\mathrm{CK}(n)^*(#1)}
\newcommand{\grckn}[1]{\mathrm{gr}_\tau^{*}\,\mathrm{CK}(n)^{*}(#1)}
\newcommand{\so}{\mathrm{SO}_m}
\newcommand{\pt}{\mathrm{pt}}

\makeatletter
\newcommand{\colim@}[2]{%
  \vtop{\m@th\ialign{##\cr
    \hfil$#1\operator@font colim$\hfil\cr
    \noalign{\nointerlineskip\kern1.5\ex@}#2\cr
    \noalign{\nointerlineskip\kern-\ex@}\cr}}%
}
\newcommand{\colim}{%
  \mathop{\mathpalette\colim@{\rightarrowfill@\textstyle}}\nmlimits@
}
\makeatother

\newtheorem{lm}{Lemma}[section]
\newtheorem{lm*}{Lemma}
\newtheorem*{tm*}{Theorem}
\newtheorem*{tms*}{Satz}
\newtheorem{tm}[lm]{Theorem}
\newtheorem{prop}[lm]{Proposition}
\newtheorem{assum}[lm]{Assumption}
\newtheorem*{prop*}{Proposition}
\newtheorem{cl}[lm]{Corollary}
\newtheorem*{cor*}{Corollary}
\theoremstyle{remark}
\newtheorem*{rk*}{Remark}
\newtheorem*{rm*}{Remark}
\newtheorem{rk}[lm]{Remark}
\newtheorem*{xm}{Example}
\theoremstyle{definition}
\newtheorem{df}{Definition}
\newtheorem*{nt}{Notation}
\newtheorem{Def}[lm]{Definition}
\newtheorem*{Def-intro}{Definition}
\newtheorem{Rk}[lm]{Remark}
\newtheorem{Ex}[lm]{Example}

\theoremstyle{plain}
\newtheorem{Th}[lm]{Theorem}
\newtheorem*{Th*}{Theorem}
\newtheorem*{Th-intro}{Theorem}
\newtheorem{Prop}[lm]{Proposition}
\newtheorem*{Prop*}{Proposition}
\newtheorem{Cr}[lm]{Corollary}
\newtheorem{Lm}[lm]{Lemma}
\newtheorem*{Conj}{Conjecture}
\newtheorem*{BigTh}{Classification of Operations Theorem  (COT)}
\newtheorem*{BigTh-add}{Algebraic Classification of Additive Operations Theorem  (CAOT)}

\tikzcdset{
arrow style=tikz,
diagrams={>={Straight Barb[scale=0.8]}}
}

\title{Morava K-theory and Rost invariant}
\author{Andrei Lavrenov, Victor Petrov}
\maketitle

\begin{abstract}
We prove that inner forms of a variety of Borel subgroups have isomorphic motives with respect to the second Morava K-theory if and only if the corresponding Tits algebras and Rost invariants coincide. This extends Panin's results on interrelationship of K-theory with Tits algebras to the case of cohomological invariants of degree $3$.
\end{abstract}

\section{introduction}

\subsection{State of the art}
The relationship between cohomological invariants and ``classical'' oriented cohomology theories: Chow groups and K-theory, plays an important role in the theory of algebraic groups. 

One of the most famous examples of such a relationship is the Milnor conjecture (proven by Voevodsky~\cite{Voev} and Orlov--Vishik--Voevodsky\cite{OVV}), which provides, in particular, a classification of quadratic forms over a field in terms of cohomological invariants. Its proof also relies on Rost's computation of Chow motive of the Pfister quadric~\cite{Rost}.

Another example of such connection can be obtained from Panin's computation of K-theory of projective homogeneous varieties~\cite{Pa}.  
It follows from Panin's results that for any split semi-simple group $G$, 
two twisted forms $X$ and $X'$ of the same (split) 
projective homogeneous variety
have isomorphic 
K-theory rings if and only if 
the Tits algebras of $E$ and $E'$ inside the Brauer group coincide (cf. Theorem~\ref{tm:mainK1}).

Yet another well-studied case is that of Chow motives of generically split varieties, for instance, for any split semi-simple group $G$ and a $G$-torsor $E$  over $\mathrm{Spec}(k)$ one can consider the corresponding twisted form of the variety of Borel subgroups $E/B$. The structure of motives of such varieties is determined by the so-called $J$-invariant, which is used in Vishik's construction of fields of $u$-invariant $2^n+1$~\cite{Vuinv} and in the solution of Serre's problem about finite subgroups in the case of $\mathsf E_8$~\cite{S16, GS10}. This invariant palys an imporatnt role in the present paper as well.

It is known that if the Chow ring of a split simple $G$ has only one generator, then the isomorphism of Chow motives of $E/B$ and $E'/B$ is equivalent to the coincidence of corresponding Tits algebras and Rost invariant~\cite[Lemmas~7.1, 7.5]{PSZ}.

The main result of the present paper is Theorem~\ref{tm:mainK2} stating that the same criterion is valid for all groups if we consider second Morava K-theory $\mathrm K(2)^*$ instead of $\mathrm{CH}^*$ (although the condition on number of generators is not valid).

\subsection{Overview of cohomological invariants}

One way to study $\mathrm H^1(k,\,G)$ for a linear algebraic group $G$ over a field $k$ consists in defining maps from this set to more computable cohomology groups (by analogy with Chern classes of vector bundles). This idea dates back to Serre, and leads to the notion of cohomological invariant. 

One of the most famous examples of cohomological invariants are Tits algebras introduced in~\cite{Tits}. These invariants take values in the Brauer group $\mathrm H^2(k,\,\mathbb Q/\mathbb Z(1))$ and play an important role in the study of classification of linear algebraic groups and corresponding homogeneous varieties. Moreover, this invariant plays a key role in the computation of K-theory of twisted flag varieties by Panin~\cite{Pa}. 

For a simple simply connected $G$, Rost constructed a cohomological invariant with values in $\mathrm H^3(k,\,\mathbb Q/\mathbb Z(2))$. Roughly speaking, Rost invariant is the first non-zero cohomological invariant for such $G$, and it generates the group of invariants with values in $\mathrm H^3(k,\,\mathbb Q/\mathbb Z(2))$~\cite{Me,KMRT}. The existence of such an invariant was previously conjectured by Serre. Garibaldi, Petrov and Semenov proved that the Rost invariant for groups of type $\mathsf E_7$ detects rationality of parabolic subgroups, proving a conjecture of Rost and solving a question of Springer~\cite{GPS}. The Rost invariant plays a crutial role in the work of Bayer and Parimala on the Hasse Principle Conjecture~\cite{BP}, and in the theory of quadratic forms, where it is known as the Arason invariant.

\subsection{Overview of oriented cohomology theories}

Oriented cohomology theories, foremost Chow groups $\mathrm{CH}^i$ and algebraic K-theory $\mathrm K_i$, became a power tool to study linear algebraic groups and corresponding projective homogeneous varieties. 

Levine and Morel constructed a universal oriented cohomology theory called algebraic cobordism $\Omega^*$, and proved that $\mathrm{CH}^*$ and $\mathrm K_0$ can be obtained as its quotients~\cite{LM}. Moreover, their result allowed one to consider other possible quotients of $\Omega^*$ corresponding to well-studied topological oriented cohomology theories, for instance, Morava K-theories $\mathrm K(n)^*$. 

Morava K-theories were initially brought to the algebraic context by Voevodsky in his program for the proof of Milnor's conjecture~\cite{VoevMor}. Recall that oriented cohomology theories are endowed with formal group laws, which contain important information about the theories. The interest to Morava K-theories can be explained by the fact that they produce all possible formal group laws over algebraically closed fields.

In more geometric terms, for a $k$-smooth $X$, we can considered $\Omega^*(X)$ as a quasi-coherent sheaf on the stack of formal groups $\mathcal M_{\mathrm fg}$. Then the stalks of $\Omega^*(X)$ over the geometric points of $\mathcal M_{\mathrm fg}$ are naturally isomorphic to $\mathrm{CH}^*(X)\otimes \overline F$ and $\mathrm K(n)^*(X)\otimes\overline F$ for algebraically closed fields $\overline F$, $n\in\mathbb N$ (and $\mathrm K_0\otimes\overline F=\mathrm K(1)^*\otimes\overline F$). 

Sechin studied operations from Morava K-theory~\cite{Sechin1,Sechin}, and Sechin--Semenov used Morava K-theories to provide new estimates on torsion of in Chow groups of quadrics~\cite{SeSe}. They also proved that the splitting of Morava K-theory motives of generically split varieties is equivalent to the triviality of corresponding cohomological invariants. The latter result is essential for the present paper.

\subsection{Statement of results}

Let $G$ be a split semi-simple algebraic group over a field $k$ of characteristic $0$, and $B$ its Borel subgroup. For $E,\,E'\in\mathrm H^1(k,\,G)$ consider the corresponding twisted forms $E/B$ and $E'/B$ of the variety of Borel subgroups $G/B$.


Let $\,_pT(E)$ denote the $p$-primary component of the subgroup of all Tits algebras of $E$ inside the Brauer group $\mathrm H^2\big(k,\,\mathbb Q/\mathbb Z(1)\big)$. Assume that $\,_pT(E)$ is trivial. Then we can define the $p$-primary component of the Rost invariant $\,_pr(E)$ in $\mathrm H^3\big(k,\,\mathbb Q/\mathbb Z(2)\big)$.

Consider the second Morava K-theory $\mathrm K(2)^*$ which is a free oriented cohomology theory in the sense of Levine--Morel corresponding to the Lubin--Tate formal group law of height $2$ over the $\mathbb Z_{(p)}$. Consider the corresponding category of pure $\mathrm K(2)$-motives, and denote by $M_{\mathrm K(2)}(X)$ the motive of $k$-smooth projective variety $X$.

We prove the following result.

\begin{tm*}[Theorem~\ref{tm:mainK2}]
Let $G$ be a split simple algebraic group over a field $k$ of characteristic $0$, $B$ its Borel subgroup, and $E$, $E'$ be two $G$-torsors over $\Spec(k)$. Then $M_{\mathrm K(2)}(E/B)$ is isomorphic to $M_{\mathrm K(2)}(E'/B)$ if and only if
$$
\,_pT(E)=\,_pT(E')\quad\text{and}\quad\langle\,_pr(E_{k(X)})\rangle =\langle\,_pr(E'_{k(X)})\rangle,
$$
where $X$ is the product of all Severi--Brauer varieties of Tits algebras in $\,_pT(E)$.
\end{tm*}

In fact, we also generalize the above theorem for the case of semi-simple $G$.


\subsection{Plan of the proof}

The main new ingredient in the proof of Theorem~\ref{tm:mainK2} is the following result.

\begin{tm*}[Theorem~\ref{tm:main}]
Let $G$ be a split semi-simple algebraic group over a field $k$, $B$ be its Borel subgroup, $E$ and $E'$ be two $G$-torsors over $\mathrm{Spec}(k)$. Then $M_{\mathrm K(2)}(E/B)$ is isomorphic to $M_{\mathrm K(2)}(E'/B)$ if and only if the former becomes split over the function field of $E'/B$ and the latter becomes split over the function field of $E/B$.
\end{tm*}

The above theorem allows us to combine the results of Sechin--Semenov~\cite{SeSe} about split $\mathrm K(2)$-motives with the Theorem of Merkurjev~\cite[Theorem~9.10]{Me} that the kernel of the map
$$
    \mathrm H^3\big(k,\,\mathbb Q/\mathbb Z(2)\big)\to
    \mathrm H^3\big(k(E/B),\,\mathbb Q/\mathbb Z(2)\big)
$$
is generated by the Rost invariant.

Indeed, assume for simplicity that $G=G^{\mathrm{sc}}$. 
By~\cite[Theorem~9.1]{SeSe}, the $\mathrm K(2)$-motive of $E/B$ is split if and only if $\,_pr(E)$ is trivial. In other words, the $\mathrm K(2)$-motive of $(E/B)_{k(E'/B)}$ is split if and only if $\,_pr(E)\in\langle\,_pr(E')\rangle$ by Merkurjev's theorem, and Theorem~\ref{tm:mainK2} follows.

In turn, to prove Theorem~\ref{tm:main} we use the Hopf-theoretic approach to Morava motives developed in~\cite{PShopf}. We should remark, however, that this approach of~\cite{PShopf} is better developed for the ``$\mathrm{mod}\ p$'' version of Morava K-theory, while in~\cite{SeSe} the authors work with $\mathbb Z_{(p)}$-integral Morava K-theory. This difference does not affect motivic isomorphisms.

For a $G$-torsor $E$ over $\mathrm{Spec}(k)$ let $H^*(E)=\mathrm K(2)^*(G)\otimes_{\mathrm K(2)^*(E)}\mathbb F_2[v_2^{\pm1}]$ be its {\it generalized} $J$-invariant of~\cite[Definition~4.6]{PShopf}. It has a natural structure of a Hopf algebra, and, moreover, any $\mathrm K(2)$-correspondence between $E/B$ and $E'/B$ defines an $H^*(E\times E')$-comodule endomorphism of $\mathrm K(2)^*(G/B)$. We compute the ranks of 
$$
\mathrm{Im}\big(\mathrm K(2)^*(E/B\times E'/B)\rightarrow\mathrm K(2)^*(G/B\times G/B)\big)
$$
and 
$$
\mathrm{End}_{H^*(E\times E')-\mathrm{comod}}\big(\mathrm K(2)^*(G/B)\big)
$$ 
over $\mathbb F_2[v_2^{\pm1}]$, and show that under the conditions of Theorem~\ref{tm:main} they coincide. This implies that the identity endomorphism of $\mathrm K(2)^*(G/B)$ is rational, and therefore the $\mathrm K(2)$-motives of $E/B$ and $E'/B$ are isomorphic.

An important step in the above proof is to show that 
$$
H^*(E)\cong H^*(E\times E')\cong H^*(E')
$$
under the conditions of Theorem~\ref{tm:main}. More generally, let us denote 
\begin{equation}
\label{irr-k2}
\underline{\mathrm K(2)}^*(X)={\mathrm K(2)}^*(X\times_k\overline k)\otimes_{{\mathrm K(2)}^*(X)}\mathbb F_2[v_2^{\pm1}]
\end{equation}
for a $k$-smooth variety $X$. We prove the following result.
\begin{tm*}[Theorem~\ref{awtf}]
Let $X$, $Y$ be $k$-smooth projective geometrically cellular geometrically connected varieties. Assume that $\mathrm K(2)$-motive of $X_{k(Y)}$ is split.

Then $(\mathrm{pr}_2)^{\mathrm K(2)}\colon{\mathrm K(2)}^*(\overline Y)\rightarrow{\mathrm K(2)}^*(\overline{X}\times\overline{Y})$ induces an isomorphism
$$
\underline{\mathrm K(2)}^*(Y)\cong\underline{\mathrm K(2)}^*(X\times Y).
$$
\end{tm*}

The proof of the above result uses only standard properties of oriented cohomology theories and induction on topological filtration, and therefore is valid for a wide class of oriented cohomology theories.

\subsection{Organization of the paper}


In Section~\ref{sec:inv} we recall the necessary background on Tits algebras and Rost invariants. In particular, for a torsor of a {\it semi-simple} simply-connected group we introduce a subgroup generated by Rost invariants corresponding to simple components. This subgroup is an analogue of the subgroup of Tits algebras in the case of invariants of degree $3$.

Section~\ref{sec:mot} is devoted to the oriented cohomology theories in the sense of Levine--Morel and corresponding categories of (pure) motives. We provide proofs of several simple facts when it is hard to find an accurate refernce in the literature. We also introduce the notion of a ring of {\it irrational elements} (see~\eqref{irr-k2}) which generalizes the notion of the $J$-invariant, and provides a convinient context for the proof of Theorem~\ref{awtf} mentioned above.

In Section~\ref{sec:mor} we recall the definition of Morava K-theories with coefficients in $\mathbb Z_{(p)}$ and $\mathbb F_p$, and prove that isomrphisms of Morava K-theory motives defined modulo $p$ can be lifted to the $p$-local ones.

In Section~\ref{sec:awt} we prove Theorem~\ref{awtf}, and then in Section~\ref{sec:res} give proof of the main results: Theorems~\ref{tm:main} and~\ref{tm:mainK2}. We also prove an analogous result for $\mathrm K(1)^*$ in Theorem~\ref{tm:mainK1}.

\subsection{Acknowledgement}

 The second-named author was supported by the Foundation for the Advancement of Theoretical Physics and Mathematics ``BASIS'' and the grant of the Government of the Russian Federation for the state support of scientific research carried out under the supervision of leading scientists, agreement 14.W03.31.0030 dated 15.02.2018.



\section{Background on cohomological invariants}
\label{sec:inv}

\subsection{Torsors}

In the present paper we work over a field $k$ of characteristic $0$. For a smooth linear algebraic group $G$ over $k$ we identify the Galois cohomology group $$\mathrm H^1(k,G)=\mathrm H^1\big(\mathrm{Gal}({\overline k}/k),\,G(\overline k)\big)$$ with the set of isomorphism classes of $G$-torsors over $\mathrm{Spec}(k)$, see~\cite{Se65} or~\cite[Chapter~VII]{KMRT}.

For a split semi-simple group $G$ consider its Borel subgroup $B$ and $E\in\mathrm H^1(k,\,G)$. Then $G/B$ is the variety of Borel subgroups of $G$ and $E/B$ is the twisted form of $G/B$ by means of the cocycle $\xi_E\in\mathrm H^1\big(k,\mathrm{Aut}(G)\big)$ coming from $E$. This variety becomes split over its function field, i.e., $(E/B)_{k(E/B)}\cong(G/B)_{k(E/B)}$.

\subsection{Tits algebras}


Let $G$ be a split semi-simple linear algebraic group over $k$, let $G^{\mathrm{sc}}$ denote its simply connected cover, and let $Z$ denote the kernel of the natural map $G^{\mathrm{sc}}\rightarrow G$. Let 
$$
\partial\colon\mathrm H^1(k,\,G)\rightarrow\mathrm H^2(k,\,Z)
$$ 
denote the connecting homomorphism in the long exact sequence of Galois cohomology assotiated to a short exact sequence $Z\rightarrowtail G^{\mathrm{sc}}\twoheadrightarrow G$.

For each character $\chi\in\mathrm{Hom}(Z,\,\mathbb G_{\mathrm m})$ we denote by 
$$
t_\chi(E)=\chi_*(\partial(E))\in\mathrm H^2(k,\,\mathbb G_{\mathrm m})=\Br(k)
$$ 
(the class of) the Tits algebra of $E$ associated with $\chi$ in the Brauer group, see~\cite{Tits} and~\cite[\S~27]{KMRT}. The map $\mathrm{Hom}(Z,\,\mathbb G_{\mathrm m})\rightarrow\Br(k)$ sending $\chi$ to $t_\chi(E)$ is a group homomorphism $\beta$ of Tits~\cite{Tits},  
in particular, the set of Tits algebras of $E$ form a subgroup of $\Br(k)$ which we will denote
\begin{equation}
\label{def-tits}
T(E)=\left\{ t_\chi(E)\mid \chi\in\mathrm{Hom}(Z,\,\mathbb G_{\mathrm m})\right\}\leq\Br(k).
\end{equation}
We will also denote by $\,_{p}T(E)$ the $p$-torsion part of $T(E)$ in $\,_{p}\Br(k)=\mathrm H^2(k,\,\mu_p)$. 

If $A$ is a central simple algebra over $k$, let $[A]$ its class in the Brauer group $\Br(k)$ and $\mathrm{SB}(A)$ the corresponding Severi--Brauer variety. The following result is proven in~\cite[Lemma~8]{Pe}.
\begin{lm}[Amitsur, Peyre]
\label{tits-ker}
For $[A_1],\ldots,[A_m]\in\mathrm H^2(k,\,\mu_p)$ the kernel of the map
$$
\mathrm H^2(k,\,\mu_p)\rightarrow\mathrm H^2\big(k\big(\mathrm{SB}(A_1)\times\ldots\times\mathrm{SB}(A_m)\big),\,\mu_p\big)
$$
coincides with the subgroup of $\mathrm H^2(k,\,\mu_p)$, generated by $[A_1],\ldots,[A_m]$.
\end{lm}

\subsection{Rost invariant}
\label{sec:rost}


For a $k$-smooth equidimensional $X$ we will consider the K-cohomology group $A^1(X,\,\mathrm K_2)$, which is equal by definition to the cohomology of the complex
$$
\bigoplus_{x\in X^{(0)}}\mathrm K_2\big(k(x)\big)\rightarrow\bigoplus_{x\in X^{(1)}}\mathrm K_1\big(k(x)\big)\rightarrow\bigoplus_{x\in X^{(2)}}\mathrm K_0\big(k(x)\big)
$$
in the middle term, where $X^{(i)}$ is the set of points of $X$ of codimension $i$, and $\mathrm K_i(F)$ denotes the algebraic K-theory of a field $F$.

Let $G$ be a simply-connected semi-simple group over $k$, and $E\in\mathrm H^1(k,\,G)$. Consider the sequence
\begin{equation}
\label{theseq}
0\rightarrow A^1(E,\,\mathrm K_2)\rightarrow A^1(G,\,\mathrm K_2)\xrightarrow{\delta_E}\mathrm H^3\big(k,\,\mathbb Q/\mathbb Z(2)\big)\rightarrow\mathrm H^3\big(k(E),\,\mathbb Q/\mathbb Z(2)\big)
\end{equation}
of~\cite[Proposition~9.2, Theorem~9.3]{Me}. We remark that $\delta_E$ is natural in $E$~\cite[Remark~8.10]{Me}.

For a simply-connected split simple algebraic group $G$ the group $A^1(G,\,\mathrm K_2)$ is {\it canonically} isomorphic to the infinite cyclic group
\begin{equation*}
A^1(G,\,\mathrm K_2)\cong\mathbb Z
\end{equation*}
by~\cite[Corollary~6.12, \S6.10]{Me}. Under this identification, for $E\in\mathrm H^1(k,\,G)$ we denote 
\begin{equation}
\label{def-rost-inv}
r(E)=\delta_E(1)
\end{equation}
{\it the Rost invariant} of $E$. More generally, let $G$ be a split simple group over $k$, $G^{\mathrm{sc}}$ its simply-connected cover, and $E\in\mathrm H^1(k,\,G)$ such that all Tits algebras of $E$ are trivial. Then there exist a pre-image $E^{\mathrm{sc}}\in\mathrm H^1(k,\,G^{\mathrm{sc}})$ of $E$, and, moreover, for $r(E^{\mathrm{sc}})$ does not depend on the choice of pre-image (cf.~\cite[Section~2]{GaPe}). We will denote $r(E^{\mathrm{sc}})$ simply by $r(E)$.

 A simply-connected split semi-simple group $G$ is the direct product of uniquely determined split simple groups $G_i$,
 $$
 G=G_1\times\ldots\times G_r.
 $$
Then for $E\in\mathrm H^1(k,\,G)$ consider the corresponding decomposition $E=E_1\times\ldots\times E_r$ for $E_i\in\mathrm H^1(k,\,G_i)$, and denote
$$
R(E)=\langle r(E_i)\mid i=1,\ldots r\rangle\leq H^3\big(k,\,\mathbb Q/\mathbb Z(2)\big)
$$
the subgroup generated by all Rost invariants corresponding to simple factors of $G$. We will also denote $\,_pR(E)$ the $p$-primary component of $R(E)$. We will need the following result. 

\begin{lm}
\label{rost-ker}
Let $G$ be a simply-connected split semi-simple group, $B$ its Borel subgroup, and $E\in\mathrm H^1(k,\,G)$. Then $R(E)$ coincides with the kernel of the map
\begin{equation}
\label{maptoeb}
\mathrm H^3\big(k,\,\mathbb Q/\mathbb Z(2)\big)\rightarrow\mathrm H^3\big(k(E/B),\,\mathbb Q/\mathbb Z(2)\big).
\end{equation}
\end{lm}
\begin{proof}
Since $E\rightarrow E/B$ has fiber $B$, 
we conclude that the field extension $k(E)\supseteq k(E/B)$ is purely transcendental, therefore the kernel of~\eqref{maptoeb} coincides with the kernel of
$$
\mathrm H^3\big(k,\,\mathbb Q/\mathbb Z(2)\big)\rightarrow\mathrm H^3\big(k(E),\,\mathbb Q/\mathbb Z(2)\big).
$$ 
Using the exactness of~\eqref{theseq}, we can identify this kernel with the image of $\delta_E$. However, the natural projections induce an isomorphism
\begin{equation*}
A^1(G,\,\mathrm K_2)\cong\bigoplus_i A^1(G_i,\,\mathrm K_2)
\end{equation*}
by~\cite[Proposition~7.6, Theorem~9.3]{Me}, therefore $\mathrm{Im}(\delta_E)$ coincides with $R(E)$.
\end{proof}

Finally, assume that $G$ is a split semi-simple group, and $E\in\mathrm H^1(k,\,G)$ such that the $p$-primary component $\,_pT(E)$ of the subgroup of Tits algebras of $E$ is trivial. Then after passing to an extension $L/k$ of degree coprime to $p$ all Tits algebras of $E_L$ vanish, and $E_L$ comes from a torsor $E^{\mathrm{sc}}$ under the action of the simply connected cover $G^{\mathrm{sc}}$ of $G$. Let $G^{\mathrm{sc}}=G_1^{\mathrm{sc}}\times\ldots\times G_r^{\mathrm{sc}}$, and $E^{\mathrm{sc}}=E_1^{\mathrm{sc}}\times\ldots\times E_r^{\mathrm{sc}}$, for $E_i^{\mathrm{sc}}\in\mathrm H^1(L,\,G^{\mathrm{sc}})$, and $r(E_i^{\mathrm{sc}})$ be corresponding Rost invariants. 

Obserbe that $r(E_i^{\mathrm{sc}})$ are invariant with respect to the action of $\mathrm{Aut}_k(L)$. Indeed, let $G^{\mathrm{ad}}$ denote the adjoint quotient of $G$ over $k$, and $G^{\mathrm{ad}}=G_1^{\mathrm{ad}}\times\ldots\times G_r^{\mathrm{ad}}$ its decomposition into prime factors in such a way that $G_i^{\mathrm{sc}}$ is the simply-connected cover of $(G_i^{\mathrm{ad}})_L$. Let $E^{\mathrm{ad}}$ denote the image of $E$ in $\mathrm H^1(k,\,G^{\mathrm{ad}})$, and decompose $E^{\mathrm{ad}}=E_1^{\mathrm{ad}}\times\ldots\times E_r^{\mathrm{ad}}$ for $E_i^{\mathrm{ad}}\in\mathrm H^1(L,\,G^{\mathrm{ad}})$. Since $(E_i^{\mathrm{ad}})_L$ are defined over $k$, they are invariant uner the action of $\mathrm{Aut}_k(L)$, in particular, $\,^\sigma E_i^{\mathrm{sc}}$ is a lift of $(E_i^{\mathrm{ad}})_L$ for any $\sigma\in\mathrm{Aut}_k(L)$, therefore $\,^\sigma r(E_i^{\mathrm{sc}})=r(E_i^{\mathrm{sc}})$.

Let $\,_pr(E_i^{\mathrm{sc}})$ denote the $p$-components of $E_i^{\mathrm{sc}}$. Then 
$$
\,_pr_i(E)=\frac{1}{[L:k]}\,\mathrm{cores}_{L/k}\big(\,_pr(E_i^{\mathrm{sc}})\big)
$$
depends only on the numbering of $G_i^{\mathrm{ad}}$, but not on the choice of $L$. We will call them the $p$-components of the Rost invariants of $E$, and also denote
\begin{equation}
\label{def-rost}
\,_pR(E)=\langle\,_pr_i(E)\mid i=1,\ldots r\rangle\leq H^3\big(k,\,\mathbb Q/\mathbb Z(2)\big)
\end{equation}
the subgroup they generate. 

Since all elements of $R(E^{\mathrm{sc}})$ are invariant with respect to the action of $\mathrm{Aut}_k(L)$, we conclude that restriction homomorphism $\mathrm{res}_{L/k}$ togethter with $\frac{1}{[L:k]}\mathrm{cores}_{L/k}$ definte mutually inversre isomorphisms between $\,_pR(E)$ and $\,_pR(E^{\mathrm{sc}})$. Moreover, $\,_pr_i(E)$ are functorial with respect to the field extensions of $k$.

\section{Background on oriented cohomology theories}
\label{sec:mot}

\subsection{Rational elements}

Let $k$ be a field of characteristic $0$ and let $A^*$ denote an oriented cohomology theory on $\mathcal S\mathsf m_k$ in the sense of Levine--Morel~\cite{LM}. 

For a projective $X\in\mathcal S\mathsf m_k$ we will use the notation $\chi\colon X\rightarrow k$ for the structure morphism and 
$$
\chi_A\colon A^*(X)\rightarrow A^*(k)
$$ 
for the push-forward map along $\chi$.

Assume that $A^*$ is generically constant. Then for any $k$-smooth irreducible variety $X$ the ring $A^*(X)$ is naturally augmented 
$$
\varepsilon_A\colon A^*(X)\xrightarrow{(\eta_X)^A}A^*(\mathrm{Spec}\,k(X))\cong A^*(\mathrm{pt})
$$ 
by the pullback map along the inclusion of the generic point 
$
\eta_X\colon\mathrm{Spec}\, k(X)\rightarrow X.
$

We will denote the augmentation ideal by $\,^+$, i.e.,
 $$
A^*(X)^+=\mathrm{Ker}\big(\varepsilon_A\colon A^*(X)\rightarrow A^*(\mathrm{pt})\big).
$$

Below we will assume that $A^*$ is a free theory~\cite[Remark~2.4.14]{LM}. For any field extension $E/k$ the functor $X\mapsto A^*(X_E)$ is an oriented cohomology theory on $\mathcal S\mathsf m_k$, and using the universal property of $A^*$ we get a natural transformation
$$
\mathrm{res}\colon A^*(X)\rightarrow A^*(X_E)
$$
which we will call {\it the restriction of scalars}. We denote $\overline k$ the algebraic closure of $k$, and for $X\in\mathcal S\mathsf m_k$ we denote $\overline X=X\times_k\overline k$.

\begin{df}
Let $X$ be a 
$k$-smooth variety. We denote
$$
\overline A^*(X)=\mathrm{Im}\left(\mathrm{res}\colon A^*(X)\rightarrow A^*(\overline X)\right)
$$
the image of $A^*(X)$ in $A^*(\overline X)$ under the restriction of scalars map and call it {\it the subring of $k$-rational} (or just {\it rational}) {\it elements}. We also call
$$
\underline A^*(X)=A^*(\overline X)\otimes_{A^*(X)}A^*(\mathrm{pt}).
$$
{\it the quotient ring of irrational elements}, although these irrational elements are {not} actual elements of $A^*(\overline X)$.
\end{df}

%
%
%
%
%
%

Recall that the $k$-scheme $X$ is called geometrically connected if $X_E$ is connected for any field extension $E/k$. In particular, for a geometrically connected $k$-smooth variety $X$ one has
\begin{equation}
\label{notation}
\mathrm{res}\big(A^*(X)^+\big)=\overline A^*(X)\cap A^*(X)^+,
\end{equation}
and we will denote the set~\eqref{notation} simply by $\overline A^*(X)^+$.

\subsection{Generalized $J$-invariant}

Following~\cite[Lemma~4.5]{PShopf} we give the following definition.

\begin{df}
If $X\in\mathcal S\mathsf m_k$ is geometrically connected, 
denote 
$$
J_A(X)=A^*(\overline X)\cdot \overline A^*(X)^+\trianglelefteq A^*(\overline X).
$$
Clearly, $\underline A^*(X)$ is a quotient ring of $A^*(\overline X)$ modulo the ideal $J_A(X)$.
\end{df}

\begin{Ex}
Let $G$ be a split semi-simple algebraic group over $k$, and $E$ be a $G$-torsor over $\mathrm{Spec}(k)$. Then $J_A(E)$ coincides with $J$ of~\cite[Lemma~4.5]{PShopf}; in particular, $J_A(E)$ is a two-sided bi-ideal.
\end{Ex}

Following~\cite[Definition~4.6]{PShopf} we also give the following definition.

\begin{df}
\label{def-H}
Let $G$ be a split semi-simple algebraic group over $k$, and $E$ be a $G$-torsor over $\mathrm{Spec}(k)$. Then we will call
$$
H_A^*(E)=\underline A^*(E).
$$
the ({\it generalized}) {\it $J$-invariant} of $E$ with respect to the theory $A^*$. Clearly, $H_A^*(E)$ coincides with $H^*$ of~\cite[Definition~4.6]{PShopf}, in particular, it is a bi-algebra.
\end{df}


\subsection{Topological filtration}

We will use the topological filtration $\tau^\bullet$ on $A^*(X)$, see~\cite[Section~1.8]{Sechin}:
$$
\tau^i A^*(X)=\bigcup_{\substack{U\text{ open}\\\mathrm{codim}_X X\setminus U \ge i}} \mathrm{Ker\ } \big(A^*(X)\rarr A^*(U)\big).
$$
We will use the following properties of this filtration (see~\cite[Proposition~1.17]{Sechin}).

\begin{lm}
\label{top-fil}
Let $A^*$ be a free theory, and $X$ a $k$-smooth variety. Then the topologival filtration $\tau^iA^*(X)$ has the following properties:
\begin{enumerate}
    \item
    $\tau^iA^*(X)\cdot\tau^jA^*(X)\subseteq\tau^{i+j}A^*(X)$ for all $i$, $j\in\mathbb N\cup0$;
    \item     $\forall\,n\geq\mathrm{dim}\,X+1$ one has $\tau^nA^*(X)=0$;
    \item
    $\tau^0A^*(X)=A^*(X)$;
    \item 
    if $X$ is connected then $\tau^1A^*(X)=A^*(X)^+$;
    \item
    If $X$ has a rational point then its class $[\mathrm{pt}]_A$ lies in $\tau^{\mathrm{dim}\,X}A^*(X)$, in particular,  
    $$
    [\mathrm{pt}]_A\cdot\tau^1A^*(X)=0.
    $$
\end{enumerate}
\end{lm}

\subsection{Motives}

Following~\cite{manin}, we consider a category of $A$-motives. For a $k$-smooth projective variety $X$ its $A$-motive will be denoted $M_A(X)$. In particular, for connected $X$, $Y\in\mathcal S\mathsf m_k$ one has
$$
\mathrm{Hom}(\big(M_A(X),\,M_A(Y)\big)=A^{\mathrm{dim}\,Y}(X\times Y),
$$
and the composition is given by
\begin{equation}
\label{comp-corr}
\beta\circ\alpha=(\mathrm{pr}_{1,3})_A\big(\mathrm{pr}_{1,2}^A(\alpha)\cdot\mathrm{pr}_{2,3}^A(\beta)\big)
\end{equation}
for $\alpha\in A^*(X\times Y)$, $\beta\in A^*(Y\times Z)$, and connected $X$, $Y$, $Z\in\mathcal S\mathsf m_k$ (see~\cite{manin}).

We will say that a $k$-smooth projective variety $X$ satisfies {\it Rost nilpotence principle}, if for any field extension $E/k$ the kernel of the restriction map
$$
\mathrm{res}\colon\mathrm{End}\big(M_A(X)\big)\rightarrow\mathrm{End}\big(M_A(X_E)\big)
$$
consists of nilpotent elements.

If $X$ is a projective homogeneous variety for a semi-simple linear algebraic group over $k$, then $X$ satisfies Rost nilpotence principle by~\cite[Corollary~4.5]{GV}.

\subsection{Split motives}

We say that the $A$-motive of a $k$-smooth projective variety $X$ is {\it split} if it is isomorphic to a direct sum of twisted $A$-motives of a point $\mathrm{Spec}(k)$.

We call a $k$-smooth projective variety $X$ {\it cellular}, if it is a disjoint
union of its subvarieties $X_i$ such that $X_i\cong\mathbb A^{n_i}$ for some $n_i\geq0$. Clearly, $A$-motives of cellular varieties are split for any $A^*$.

We will use the K\"unneth isomorphism 
\begin{equation}
\label{kunneth}
A^*(X)\otimes A^*(Y)\cong A^*(X\times Y)
 \end{equation} 
given by $x\otimes y\mapsto\mathrm{pr}_1^A(x)\cdot\mathrm{pr}_2^A(Y)$ for $k$-smooth projective $X$ and $Y$ such that $A$-motive of $Y$ is split (cf.~\cite{NZ}).

We will also use the following well-known result.

\begin{lm}
\label{linear-algebra}
Let $X$ and $Y$ be $k$-smooth projective cellular varieties. Denote 
$$
M= A^*(X)\ \text{ and }\ N=A^*(Y).
$$
Then for $\alpha\in A^*(X\times Y)$ the assignment
$$
\alpha\mapsto\alpha_\star=(\mathrm{pr}_2)_{A}\big(\alpha\cdot\mathrm{pr}_1^{A}(-)\big)\colon M\rightarrow N
$$
defines a bijection between $A^*(X\times Y)$ and $\mathrm{Hom}_{A^*(k)}\big(M,\,N\big)$.
\end{lm}

%

It is also well-known that the assignment $\alpha\mapsto\alpha_\star$ respects the composition of correspondences~\eqref{comp-corr}, i.e., 
$$
(\beta\circ\alpha)_\star=\beta_\star\circ\alpha_\star.
$$

We will also need the following simple observations.

\begin{lm}
\label{lm:pur-trans}
Let $X$ be a projective homogeneous variety for a semi-simple algebraic group over $k$.
\begin{enumerate}[{\rm (i)}]
\item
Let $Y$ be a $k$-smooth projective irreducible variety, such that $A$-motive of $Y$ is split, and assume that $A$-motive of $X_{k(Y)}$ is split. Then the $A$-motive of $X$ is also split.
\item
Let $E/k$ be a purely transcendental field extension, and assume that the $A$-motive of $X_E$ is split. Then the $A$-motive of $X$ is also split.
\item
Let $Y$ be a $k$-smooth projective geometrically irreducible variety, and assume that the $A$-motive of $X\times Y$ is split. Then the $A$-motive of $X$ is also split.
\end{enumerate}
\end{lm}

\begin{proof}
\begin{enumerate}[{\rm (i)}]
\item
By~\cite[Lemma~7.8]{SeSe} it is enough to show that the restriction map
\begin{equation*}
\mathrm{res}\colon A^*(X)\rightarrow A^*(X_{k(Y)})
\end{equation*}
is surjective. Consider the diagram
$$
\begin{tikzcd}
A^*(X)\ar{rd}[swap]{a\mapsto a\otimes 1}\ar{r}{\mathrm{pr}_1^A}&A^*(X\times Y)\ar[->>]{r}{(\mathrm{id}\times\eta_{Y})^A}&A^*(X_{k(Y)})\\
&A^*(X)\otimes A^*(Y)\ar{u}{\cong}\ar{ru}[swap]{a\otimes b\mapsto\mathrm{res}(a)\varepsilon(b)}&
\end{tikzcd}
$$
where the vertical arrow is the K\"unneth isomorphism, and the right horizontal arrow is surjective due to the localization axiom (and~\cite[Corollary~2.13]{GV}). The claim follows from a simple diagram chase.
\item
Take $Y=\mathbb P^n$ in i).
\item
Using the diagram
$$
\begin{tikzcd}
A^*(X\times Y)\ar[->>]{d}{\mathrm{id}\times\eta_Y}\ar{r}{\cong}&A^*(\overline X\times\overline Y)\ar[->>]{d}{\mathrm{id}\times\eta_{\,\overline Y}}\\
A^*(X_{k(Y)})\ar{r}{\mathrm{res}}&A^*(\overline X_{\,\overline k(\overline Y)})
\end{tikzcd}
$$
we conclude that $\mathrm{res}\colon A^*(X_{k(Y)})\rightarrow A^*(\overline X)$ is surjective. By~\cite[Lemma~7.8]{SeSe} we conclude that the $A$-motive of $X_{k(Y)}$ (and therefore of $X_{k(X\times Y)}$) is split. The claim now follows from i).
\end{enumerate}
\end{proof}

\subsection{Galois action}

Let $\iota\colon k\hookrightarrow L$ be a Galois extension of the base field, and $X\in\Sm$. Each element $\gamma\in\mathrm{Gal}(L/k)$ induces a ring automorphism $\big(\mathrm{id}_X\times\mathrm{Spec}(\gamma)\big)^A$ of $A^*(X_L)$ which we denote by $x\mapsto\!\,^\gamma x$.  For a rational element $x\in\overline{A}^{\,*}(X_L)$ we clearly have $\!\,^\gamma x=x$ for all $\gamma\in\mathrm{Gal}(L/k)$.

If $L/k$ is finite, we can define the {\it corestriction map}
$$
\mathrm{cores}=\mathrm{cores}_{L/k}=\big(\mathrm{id}\times\mathrm{Spec}(\iota)\big)_A\colon A^*(X_L)\rightarrow A^*(X).
$$
The composition map
$$
\xymatrix{
A(X_L)\ar[r]^{\mathrm{cores}}& A(X)\ar[r]^{\mathrm{res}}& A(X_L)
}
$$
is given by $x\mapsto\sum\limits_{\gamma\in\mathrm{Gal}(L/k)}\!\,^\gamma x$. 

In particular, for a Galois-invariant element $x\in A^*(X_L)$ the element $[L:k]\cdot x$ is rational. The following sharper statement is proven in~\cite[Proposition~7]{Hau} for $A^*=\mathrm{CH}^*$, however the proof works for any $A^*$.

\begin{lm}[Haution]
\label{cl2}
Let $L/k$ be a finite Galois extension of degree $m=[L:k]$, $X$ is a $k$-smooth projective variety, and assume that $m$ is not a zero-divisor in $A^*(X_L)$. Then $x\in A^*(X_L)$ is $k$-rational if and only if it is Galois-invariant and its image in $A^*(X_L)/m$ is $k$-rational.
\end{lm}

It is also easy prove the following fact (cf.~\cite[Corollary~4.8]{SZ}).
\begin{lm}
\label{galois-invariant}
Let $X$ and $Y$ be projective homogeneous varieties, and let $L/k$ be a Galios extension such that $X_L$ and $Y_L$ are cellular. Then $\alpha\in A^*(X_L\times Y_L)$ is invariant under the action of $\mathrm{Gal}(L/k)$ if and only if $\alpha_*$ respects the $\mathrm{Gal}(L/k)$-action.
\end{lm}
\begin{proof}
For $\gamma\in\mathrm{Gal}(L/k)$ and $x\in A^*(X_L)$  one has
\begin{align}
\label{action-for-dummies}
\,^\gamma\big(\alpha_*(x)\big)=(\mathrm{pr}_2)_{A}\big(\,^\gamma\alpha\cdot\mathrm{pr}_1^{A}(\,^\gamma x)\big).
\end{align}
In particular, if $\alpha$ is stable under the action of $\mathrm{Gal}(L/k)$, then clearly $\,^\gamma\big(\alpha_*(x)\big)=\alpha_*(\,^\gamma x)$.

Assume now that $\alpha_*$ respects the action of $\mathrm{Gal}(L/k)$, and denote $\theta=\,^\gamma\alpha-\alpha$. Then~(\ref{action-for-dummies}) impies that
$$
(\mathrm{pr}_2)_{A}\big(\theta\cdot\mathrm{pr}_1^{A}(x)\big)=0
$$
for all $x\in A^*(X_L)$. In other words, $\theta_\star$ is a zero morphism, and using Lemma~\ref{linear-algebra} we conclude that $\theta=0$, i.e., $\,^\gamma\alpha=\alpha$.
\end{proof}

\section{Morava K-theory}
\label{sec:mor}

\subsection{Definition and simplest properties}

For a prime $p$ consider the universal $p$-typical formal group law $F_{\mathrm{BP}}$ defined over the ring $V\cong\mathbb Z_{(p)}[v_1,\,v_2,\ldots]$, where $v_k$ are {\it Hazewinkel's generators}, see~\cite[Theorem~A2.1.25]{Rav}. Fixing any $n$, inverting $v_n$, and sending $v_k$ to $0$ for $k\neq n$ we obtain a formal group $F_{\mathrm K(n)}$ over the ring $\mathbb Z_{(p)}[v_n^{\pm1}]$, 
where $v_n$ has degree $1-p^n$. We call the corresponding free theory
$$
\mathrm K(n)^*(-)=\mathrm K(n)^*\big(-;\,\mathbb Z_{(p)}[v_n^{\pm1}]\big)=\Omega^*\otimes_{\mathbb L}\mathbb Z_{(p)}[v_n^{\pm1}]
$$ 
the ($n$-th) Morava K-theory. 

One can also consider a version of Morava K-theory with coefficients in a $\mathbb Z_{(p)}$-algebra $R$, 
$$
\mathrm{K}(n)^*\big(-;\,R[v_n^{\pm1}]\big)=\Omega^*\otimes_{\mathbb L}R[v_n^{\pm1}]=\mathrm{K}(n)^*\big(-;\,\mathbb Z_{(p)}[v_n^{\pm1}]\big)\otimes R,
$$
e.g., for $R=\mathbb Z/p^k$.

For $n=1$ the Artin--Hasse exponent~\cite[Chapter~7, Section~2]{Ro} provides an isomorphism between the formal group law $F_{\mathrm K(1)}$ and the multiplicative formal group law $x+y-v_1xy\in\mathbb Z_{(p)}[v_1^{\pm1}]$. Therefore $\mathrm K(1)^*$ is isomorphic to $\mathrm K_0\otimes\mathbb Z_{(p)}[v_1^{\pm1}]$ as a presheaf of rings.

It is sometimes convenient to work with a presheaf of $\mathbb Z/(p^n-1)$-graded, rather than $\mathbb Z$-graded rings
$$
\mathrm K(n)^{[*]}(-;\,R)=\mathrm K(n)^*(-;\,R[v_n^{\pm1}])/(v_n-1).
$$
Observe that this presheaf is endowed with push-forwards in the sense of~\cite[Definition~1.1.2\,(D2)]{LM}, and satisfies all axioms of oriented cohomology theories (A1), (A2), (PB) and (EH) of~\cite[Definition~1.1.2]{LM}. Moreover, one can consider the corresponding category of motives with $\mathrm{Hom}$-sets given by
$$
\mathrm{Hom}\left(\mathcal M_{\mathrm K(n)^{[*]}}(X),\,\mathcal M_{\mathrm K(n)^{[*]}}(Y)\right)=\mathrm K(n)^{[\mathrm{dim}\,Y]}(X\times Y;\,R).
$$
Clearly, the natural map
\begin{equation}
\label{dagger}
\mathrm K(n)^{\mathrm{dim}\,Y}(X\times Y;\,R[v_n^{\pm1}])\rightarrow\mathrm K(n)^{[\mathrm{dim}\,Y]}(X\times Y;\,R)
\end{equation}
is an isomorphism, therefore the category of $\mathrm K(n)^{[*]}$-motives is isomorphic to the the usual category of $\mathrm K(n)^{*}$-motives. However, it is sometimes easier to work with $R$- rather than $R[v_n^{\pm1}]$-modules (cf. Theorem~\ref{maranda} below).

\subsection{Morava motives with integral and finite coefficients}

Most of the arguments in this subsection are taken from~\cite{SZ}. 
We will prove the following result.

\begin{tm}
\label{int=modp}
Let $X$ and $Y$ be projective homogeneous varieties such that their $\mathrm K(n)(-;\,\mathbb F_p[v_n^{\pm1}])$-motives are isomorphic. Then their $\mathrm K(n)(-;\,\mathbb Z_{(p)}[v_n^{\pm1}])$-motives are also isomorphic.
\end{tm}

We start with the following simple observation.

\begin{lm}
\label{mod-p-k}
Let $X$ and $Y$ be projective homogeneous varieties such that their $\mathrm K(n)(-;\,\mathbb F_p[v_n^{\pm1}])$-motives are isomorphic. Then their $\mathrm K(n)(-;\,\mathbb Z/p^k[v_n^{\pm1}])$-motives are also isomorphic for any $k$.
\end{lm}
\begin{proof}
We will argue by induction on $k$. Using induction hypothesis, take rational elements $\alpha\in\overline{\mathrm K(n)}^{\,\mathrm{dim}\,Y}(X\times Y;\,\mathbb Z/p^k[v_n^{\pm1}])$ and $\beta\in\overline{\mathrm K(n)}^{\,\mathrm{dim}\,X}(Y\times X;\,\mathbb Z/p^k[v_n^{\pm1}])$ such that
$$
\alpha\circ\beta=\Delta_Y+p\,\zeta_1\ \text{ and }\ \beta\circ\alpha=\Delta_X+p\,\xi_1. 
$$ 
However,
$$
(\Delta_Y+p^t\zeta_t)^{\circ p}=\Delta_Y+p^{t+1}\zeta_{t+1}
$$
for some $\zeta_{t+1}$, therefore $(\Delta_Y+p\,\zeta_1)^{\circ p^{k-1}}=\Delta_Y+p^k\zeta_k=\Delta_Y$, and similarly $(\Delta_X+p\,\xi_1)^{\circ p^{k-1}}=\Delta_X$. Then $\alpha$ and $\beta\circ(\alpha\circ\beta)^{\circ(p^{k-1}-1)}$ are rational mutually inverse isomorphisms.
\end{proof}

We will also need the following result (see~\cite[Theorem~30.14 and Definitions~30.12, 23.1--4]{CR}).
\begin{tm}[Maranda]
\label{maranda}
Let $\Gamma$ be a finite group, and $U$ and $V$ be $\mathbb Z_{(p)}[\Gamma]$-modules which are free of finite rank as $\mathbb Z_{(p)}$-modules. Assume that there exist isomorphisms of $(\mathbb Z/p^k)[\Gamma]$-modules $\alpha_k\colon U/p^k\stackrel{\cong}{\rightarrow} V/p^k$ for all $k\geq1$, and such that $\alpha_k/p^{k-1}=\alpha_{k-1}$. Then there exists an isomorphism $\alpha\colon U\xrightarrow{\cong}V$ of $\mathbb Z_{(p)}[\Gamma]$-modules which lifts $\alpha_k$.
\end{tm}

We are ready to prove Theorem~\ref{int=modp}.

\begin{proof}[Proof of Theorem~\ref{int=modp}]
Let $L/k$ be a finite Galois extension which splits both $X$ and $Y$, and $\Gamma=\mathrm{Gal}(L/k)$. Let us denote $U=\mathrm K(n)^{[*]}(X_L;\,\mathbb Z_{(p)})$ and $V=\mathrm K(n)^{[*]}(Y_L;\,\mathbb Z_{(p)})$. They are free $\mathbb Z_{(p)}$-modules of finite rank.

By Lemma~\ref{mod-p-k} we can find rational mutually inverse isomorphisms $\alpha_k$ and $\beta_k$ of $\mathrm K(n)(-;\,(\mathbb Z/p^k)[v_n^{\pm1}])$-motives of $X$ and $Y$. Since these elements are rational, they are stable under the action of the Galois group. Then by Lemma~\ref{galois-invariant} we conclude that $(\alpha_k)_\star$ and $(\beta_k)_\star$ define $(\mathbb Z/p^k[v_n^{\pm1}])[\Gamma]$-isomorphisms between $\mathrm K(n)^{[*]}(X_L;\,\mathbb Z/p^k)$ and $\mathrm K(n)^{[*]}(Y_L;\,\mathbb Z/p^k)$, which induce $(\mathbb Z/p^k)[\Gamma]$-isomorphisms between  $U/p^k$ and $V/p^k$. 

Invoking Theorem~\ref{maranda} we obtain an isomorphism of $U\cong V$ of $\mathbb Z_{(p)}[\Gamma]$-modules. By Lemma~\ref{linear-algebra}  we obtain $\alpha\in\mathrm K(n)^{[*]}(X_L\times Y_L;\,\mathbb Z_{(p)})$ and $\beta\in\mathrm K(n)^{[*]}(Y_L\times X_L;\,\mathbb Z_{(p)})$ which define an isomorphism of $\mathrm K(n)(-;\,\mathbb Z_{(p)})$-motives of $X_L$ and $Y_L$. Observe that $\alpha$ is homogeneous of codimension $\mathrm{dim}\,Y$ since it lifts all $\alpha_k$. Therefore using~\eqref{dagger} we can consider $\alpha$ as an element of $\mathrm K(n)^{\mathrm{dim}\,Y}(X_L\times Y_L;\,\mathbb Z_{(p)}[v_n^{\pm1}])$, and similarly for $\beta$. 

By Lemma~\ref{galois-invariant} $\alpha$ and $\beta$ are Galois-invariant. Moreover, if $m=p^r\cdot d=[L\colon k]$, $(p,\,d)=1$, then the image of $\alpha$ in $\mathrm K(n)^{\mathrm{dim}\,Y}(X_L\times Y_L;\,(\mathbb Z_{(p)}/m)[v_n^{\pm1}])=\mathrm K(n)^{\mathrm{dim}\,Y}(X_L\times Y_L;\,(\mathbb Z/p^r)[v_n^{\pm1}])$ is rational by construction, and the same is true for $\beta$. Applying Corollary~\ref{cl2} we conclude that $\alpha$ and $\beta$ themself are rational.
\end{proof}

Due to the Theorem~\ref{int=modp}, in the sequel we will not distinguish between the categories of $\mathrm K(n)^*(-;\,\mathbb Z_{(p)}[v_n^{\pm1}])$- and $\mathrm K(n)^*(-;\,\mathbb F_{p}[v_n^{\pm1}])$-motives.

\section{Argument with topological filtration}
\label{sec:awt}

In this section we will prove the following theorem.

\begin{tm}
\label{awtf}
Let $A^*$ be a free theory, and $X$, $Y$ be $k$-smooth projective geometrically cellular geometrically connected varieties. Assume that $A$-motive of $X_{k(Y)}$ is split.

Then $(\mathrm{pr}_2)^A\colon A^*(\overline Y)\rightarrow A^*(\overline{X}\times\overline{Y})$ induces an isomorphism
$$
\underline A^*(Y)\cong\underline A^*(X\times Y).
$$
\end{tm}

The proof is based on the series of lemmas. First of all, we will show that the induced map from $\underline A^*(Y)$ to $\underline A^*(X\times Y)$ is indeed well-defined.

\begin{lm}
\label{well-def}
Let $A^*$ be a free theory, and $X$, $Y$ be $k$-smooth geometrically connected varieties. Then the natural map
$$
(\mathrm{pr}_2)^A\colon A^*(\overline Y)\rightarrow A^*(\overline X\times\overline Y)
$$
induces a well-defined map
$$
\underline{\mathrm{pr}_2^A}\colon\underline A^*(Y)\rightarrow \underline A^*(X\times Y).
$$
\end{lm}

\begin{proof}
Using the diagram
$$
\begin{tikzcd}
A^*(Y)\ar{d}{\mathrm{res}}\ar{r}{\mathrm{pr}_2^A}&A^*(X\times Y)\ar{d}{\mathrm{res}}\\
A^*(\overline Y)\ar{r}{\mathrm{pr}_2^A}&A^*(\overline X\times\overline Y)
\end{tikzcd}
$$
we conclude that 
$$
\mathrm{pr}_2^A\big(\overline A^*(Y)\big)\subseteq\overline A^*(X\times Y).
$$
Moreover, since $\mathrm{pr}_2\colon X\times Y\rightarrow Y$ induces a morphism on generic points, we conclude that 
$$
\mathrm{pr}_2^A\big(\overline A^*(Y)^+\big)\subseteq\overline A^*(X\times Y)^+.
$$
In other words, the composite map
$$
A^*(\overline Y)\rightarrow A^*(\overline X\times\overline Y)\rightarrow \underline A^*(\overline X\times\overline Y)
$$
sends $\overline A^*(Y)^+$ to zero, and therefore factors through $\underline A^*(Y)^+$.
\end{proof}

Now we will prove that the induced map $\underline{\mathrm{pr}_2^A}$ is injective and surjective. The following simple observation will be the main ingredient in the proof of both statements.

\begin{lm}
\label{rat-on-prod}
In the notation of Theorem~\ref{awtf}, for any $a\in A^*(\overline X)$ there exists a rational element
$$
A=a\otimes1+\sum_{i=1}^ku_i\otimes v_i\in\overline A^*(X\times Y)
$$
for some $k\in\mathbb N$, $u_i\in A(\overline X)$ and $v_i\in A^*(\overline Y)^+$.
\end{lm}

\begin{proof}
Since the $A$-motive of $X_{k(Y)}$ is split, the restriction of scalars induces an isomorphism
$$
\mathrm{res}\colon A^*(X_{k(Y)})\xrightarrow{\cong}A^*(\overline X).
$$
Using the localisation property~\cite[Definition~4.4.6]{LM} (and~\cite[Corollary~2.13]{GV}), we conclude that 
$$
(\mathrm{id}\times\eta_{\,Y})^A\colon A^*(X\times Y)\rightarrow A^*(X_{k(Y)})
$$
is surjective, and therefore any $a\in A^*(\overline X)$ admits a lift $\widehat a\in A^*(X\times Y)$. Then
$$
A=\mathrm{res}(\widehat a)\in A^*(\overline X)\otimes A^*(\overline Y)
$$
clearly has the required form.
\end{proof}

Now we will prove the injectivity of the induced map. Apriori, if an element $y\in A^*(\overline Y)$ is mapped to zero in $\underline A^*(X\times Y)$ this implies that 
\begin{equation}
\label{two-reductions}
1\otimes y = \sum C_i\cdot R_i
\end{equation}
for some $C_i\in A^*(\overline X\times\overline Y)$, and $R_i\in\overline A^*(X\times Y)^+$. We will do two reductions. First, in Lemma~\ref{inj-first-red} below we will find a new decomposition of $1\otimes y$ of the form~\eqref{two-reductions} where all $C_i$ come from $A^*(\overline Y)$. Next, in Lemma~\ref{is-inj} we will find yet another decomposition of $1\otimes y$ of the form~\eqref{two-reductions} where $R_i$ come from $\overline A^*(Y)^+$ as well.

\begin{lm}
\label{inj-first-red}
In the notation of Theorem~\ref{awtf}, let $y\in A^*(\overline Y)$, and assume that the image of $1\otimes y$ in $\underline A^*(X\times Y)$ is zero. Then
$$
1\otimes y = \sum_{i=1}^k (1\otimes c_i)\cdot R_i
$$
for some $k\in\mathbb N$, $c_i\in A^*(\overline Y)$, and $R_i\in\overline A^*(X\times Y)^+$.
\end{lm}

\begin{proof}
Since the image of $1\otimes y$ in $\underline A^*(X\times Y)$ is zero, we conclude that 
\begin{equation}
\label{base=0}
1\otimes y = \sum_{j=1}^m a_j\otimes b_j\cdot S_j
\end{equation}
for some $m\in\mathbb N$, $a_j\in A^*(\overline X)$, $b_j\in A^*(\overline Y)$, and $S_j\in\overline A^*(X\times Y)^+$. We will prove by induction on $s\in\mathbb N\cup 0$ that
\begin{equation}
\label{ind-conj}
1\otimes y = \sum_{i=1}^k (1\otimes c_i)\cdot R_i + \sum_{j=1}^m a_j\otimes b_j\cdot S_j    
\end{equation}
for some $k$, $m\in\mathbb N$, $a_j\in A^*(\overline X)$, $c_i\in A^*(\overline Y)$, $R_i$, $S_j\in\overline A^*(X\times Y)^+$, and $b_j\in\tau^sA^*(\overline Y)$. The base of induction $s=0$ is~\eqref{base=0}.

Assume that~\eqref{ind-conj} is verified. For each $a_j$ there exists a rational element
$$
A_j=a_j\otimes1+\sum_{i=1}^{k_j}u_i^{(j)}\otimes v_i^{(j)}\in\overline A^*(X\times Y)
$$
for some $k_j\in\mathbb N$, $u_i^{(j)}\in A(\overline X)$ and $v_i^{(j)}\in A^*(\overline Y)^+$ by Lemma~\ref{rat-on-prod}. Then
$$
1\otimes y = \sum_{i=1}^k (1\otimes c_i)\cdot R_i + \sum_{j=1}^m(1\otimes b_j)\cdot A_j\cdot S_j-\sum_{j=1}^m\sum_{i=1}^{k_j}u_i^{(j)}\otimes(v_i^{(j)}b_j)\cdot S_j.
$$
Since $v_i^{(j)}b_j\in\tau^{s+1}A^*(\overline X)$ and $A_j\cdot S_j\in\overline A^*(X\times Y)^+$, the claim follows.
\end{proof}

\begin{lm}
\label{is-inj}
In the notation of Theorem~\ref{awtf}, the map
$$
\underline{\mathrm{pr}_2^A}\colon\underline A^*(Y)\rightarrow\underline A^*(X\times Y)
$$
of Lemma~\ref{well-def} is injective.
\end{lm}

\begin{proof}
Take $y\in A^*(\overline Y)$ such that the image of $1\otimes y$ in $\underline A^*(X\times Y)$ is zero. We have to show that the image of $y$ in $\underline A^*(Y)$ is also zero. In other words, we have to show that
$$
1\otimes A^*(\overline Y)\cap J_A(X\times Y)\subseteq 1\otimes J_A(Y).
$$
Using induction on $s\in\mathbb N\cup0$ we will show that
\begin{equation}
\label{ind-down}
1\otimes\tau^sA^*(\overline Y)\cap J_A(X\times Y)\subseteq 1\otimes J_A(Y).
\end{equation}
The base of induction is $s=\mathrm{dim}\,Y+1$. Next, assume that~\eqref{ind-down} holds. Take $y\in\tau^{s-1}A^*(\overline Y)$ such that the image of $1\otimes y$ in $\underline A^*(X\times Y)$ is zero. Then by Lemma~\ref{inj-first-red},
\begin{equation}
\label{y=cR}
1\otimes y = \sum_{i=1}^k (1\otimes c_i)\cdot R_i
\end{equation}
for some $k\in\mathbb N$, $c_i\in A^*(\overline Y)$, and $R_i\in\overline A^*(X\times Y)^+$.

Since $\overline X$ is cellular, it has a rational point. Denote its class in $A^*(\overline X)$ by $[\mathrm{pt}]$. Then there exists a rational element
\begin{equation}
\label{P=}
P=[\mathrm{pt}]\otimes 1+\sum_{i=1}^ku_i\otimes v_i\in\overline A^*(X\times Y)
\end{equation}
for some $k\in\mathbb N$, $u_i\in A(\overline X)$ and $v_i\in A^*(\overline Y)^+$ by Lemma~\ref{rat-on-prod}.

Observe that 
\begin{equation}
\label{first-red}
(\mathrm{pr}_2)_A\left(1\otimes y\cdot P\right)=(\mathrm{pr}_2)_A\left(\sum_{i=1}^k\mathrm{pr}_2^A(c_i)\cdot R_i\cdot P\right)=\sum_{i=1}^k c_i\cdot(\mathrm{pr}_2)_A\left( R_i\cdot P\right)
\end{equation}
by~\eqref{y=cR}. Clearly, $(\mathrm{pr}_2)_A(R_i\cdot P)$ is rational. Applying~\cite[Lemma~2.15\,(iii)]{GV} to the square
$$
\begin{tikzcd}
X_{k(Y)}\ar{d}{\chi}\ar{r}{\mathrm{id}\times\eta_{Y}}&X\times Y\ar{d}{\mathrm{pr}_2}\\
k(Y)\ar{r}{\eta_{Y}}&Y,
\end{tikzcd}
$$
we obtain
$$
\varepsilon_A\big((\mathrm{pr}_2)_A(R_i\cdot P)\big)=\chi_A\big((\mathrm{id}\times\eta_{\,Y})^A(R_i)\cdot[\mathrm{pt}]\big).
$$
Since  $(\mathrm{id}\times\eta_{\,Y})^A$ sends $A^*(X\times Y)^+$ to $A^*(X_{k(Y)})^+$, see Lemma~\ref{obvi}\,(4), we have 
$$
(\mathrm{id}\times\eta_{\,Y})^A(R_i)\cdot[\mathrm{pt}]=0,
$$
cf. Lemma~\ref{top-fil}\,(5). In other words, we get from~\eqref{first-red} that 
\begin{equation}
\label{pr-in-J}
(\mathrm{pr}_2)_A\left(1\otimes y\cdot P\right)\in J_A(Y).
\end{equation}

On the other hand,
$$
1\otimes y\cdot P=[\mathrm{pt}]\otimes y+\sum_{i=1}^ku_i\otimes1\cdot1\otimes(v_iy)
$$
by~\eqref{P=}, where  
$$
1\otimes(v_iy)\in1\otimes\tau^sA^*(\overline Y)\cap J_A(X\times Y)\subseteq 1\otimes J_A(Y)
$$ 
by the induction conjecture~\eqref{ind-down}. Then
\begin{equation}
\label{pr-y-in-J}
(\mathrm{pr}_2)_A\left(1\otimes y\cdot P\right)=y+\sum_{i=1}^k\chi_A(u_i)\cdot (v_iy)\in y+ J_A(Y),
\end{equation}
  and combining~\eqref{pr-y-in-J} with~\eqref{pr-in-J}, we conclude that $y\in J_A(Y)$. The claim follows.
\end{proof}


It remains to prove the following result.

\begin{lm}
\label{is-surj}
In the notation of Theorem~\ref{awtf}, the map
$$
\underline{\mathrm{pr}_2^A}\colon\underline A^*(Y)\rightarrow\underline A^*(X\times Y)
$$
of Lemma~\ref{well-def} is surjective.
\end{lm}

\begin{proof}

It is sufficient to show that the composite map
$$
A^*(\overline Y)\xrightarrow{\mathrm{pr}_2^A}A^*(\overline X\times\overline Y)\rightarrow\underline A^*(X\times Y)
$$
is surjective, i.e., that $A^*(\overline X\times \overline Y)$ is generated by the set
\begin{equation}
\label{module}
J_A(X\times Y)\cup \big(1\otimes A^*(\overline Y)\big)
\end{equation}
as an $A^*(k)$-module. Let $M$ denote the $A^*(k)$-submodule of $A^*(\overline X\times\overline Y)$ generated by the set~\eqref{module}.  We will prove by induction on $s\in\mathbb N\cup0$ that
\begin{equation}
\label{induction}
A^*(\overline X)\otimes\tau^sA^*(\overline Y)\subseteq M.
\end{equation}
The base of induction $s=\mathrm{dim}\,Y+1$ is clear. Next, assume that~\eqref{induction} holds. 

We will show that for any $a\in A^*(\overline X)$ and any $b\in\tau^{s-1}A^*(\overline Y)$ the element $a\otimes b$ belongs to $M$. 

If $a=1$, the claim follows, therefore we may assume that $a\in A^*(\overline X)^+$. There exists a rational element
$$
A = a\otimes1-\sum_{i=1}^{k}u_i\otimes v_i\in \overline A^*(X\times Y)
$$
for some $k\in\mathbb N$, $u_i\in A^*(\overline X)$, $v_i\in A^*(\overline Y)^+$ by Lemma~\ref{rat-on-prod}. Clearly, $A\in \overline A^*(X\times Y)^+$, and therefore 
\begin{equation}
    \label{argument}
a\otimes b\equiv
\left(1\otimes b\right)\cdot\sum_{i=1}^{k}u_i\otimes v_i\mod J_A(X\times Y).
\end{equation}
However, since $v_i\in A^*(\overline Y)^+$, we conclude that $bv_i\in\tau^{s}A^*(\overline Y)$, and therefore the right hand side of~\eqref{argument} lies in $M$ by the induction conjecture~\eqref{induction}. Now the claim follows.  \end{proof}

\begin{proof}[Proof of Theorem~\ref{awtf}]
Now the claim of Theorem~\ref{awtf} is a direct consequence of Lemmas~\ref{well-def},~\ref{is-inj} and~\ref{is-surj}. 
\end{proof}

We will finish this section with the following useful partial case of Theorem~\ref{awtf}.

\begin{cl}\label{cl:Jinv}
Let $A^*$ be a free theory and $G$, $G'$ be {\rm(}smooth connected{\rm)} split semi-simple  groups over a field $k$, $B$, $B'$ their Borel subgroups. Let $E$ be a $G$-torsor and $E'$ be a $G'$-torsor over $\mathrm{Spec}(k)$. Let $H^*_A(E\times E')$ and $H^*_A(E')$ denote the generalized $J$-invariants {\rm(}see Definition~\ref{def-H}{\rm)}.

Assume that $A$-motive of $E/B$ splits over the function field of $E'/B'$. Then the natural map
$$
A^*(G'/B')\xrightarrow{x\,\mapsto \,1\otimes x}A^*(G/B)\otimes A^*(G'/B')
$$
induces a bijection
$$
H^*_A(E')\xrightarrow{\cong} H^*_A(E\times E').
$$
\end{cl}

\begin{proof}
One only has to observe that 
$$
H_A^*(E')=\underline A^*(E'/B')
$$
by~\cite[Lemma~5.3]{PShopf}, and similarly $H_A^*(E\times E')=\underline A^*(E/B\times E'/B')$, and apply Theorem~\ref{awtf}.
\end{proof}

\section{Morava motives of full flag varieties}
\label{sec:res}

\subsection{General result}

Following~\cite[Assumption~5.1]{PShopf}, we make the following assumption on the theory $A^*$.

\begin{assum}
\label{assum}
In this section we assume $A^*$ is a free theory such that every finitely generated $A^*(k)$-module is projective.
\end{assum}

We will prove the following result.

\begin{tm}\label{tm:main} Let $G$ be a split semi-simple algebraic group over a field $k$, $B$ be its Borel subgroup, $E$ and $E'$ be two $G$-torsors over $\mathrm{Spec}(k)$. Then $M_A(E/B)$ is isomorphic to $M_A(E'/B)$ if and only if the former becomes split over the function field of $E'/B$ and the latter becomes split over the function field of $E/B$.
\end{tm}

\begin{proof}
Since $E/B$ (resp. $E'/B$) becomes trivial over the function field of $E/B$ (resp. $E'/B$), the ``only if'' part is clear.

Now assume that the $A$-motives of $(E/B)_{k(E'/B)}$ and $(E'/B)_{k(E/B)}$ are split. 
We can consider $G/B$ as a $G\times G$-variety with the trivial action of the second (resp., first) component, and $X=E/B$ (resp., $Y=E'/B$) as its $(E\times E')$-twisted form.

We will denote the rank of $A^*(k)$-module by $\rnk$. On the one hand, the rank of $\overline A^*(X\times Y)$ over $A^*(k)$ is equal to
\begin{equation}
\label{one-hand}
\rnk\overline A^*(X\times Y)=\frac{\rnk A^*(G/B\times G/B)}{\rnk H^*_A(E\times E')}=\frac{\rnk A^*(G/B)^2}{\rnk H^*_A(E\times E')}
\end{equation}
by~\cite[Lemma~5.4]{PShopf} and K\"unneth formula~\eqref{kunneth}. 

On the other hand, each element $\alpha$ from $\overline A^*(X\times Y)$ defines a morphism of $H_A^*(E\times E')$-comodules
$$
\alpha_\star=(\mathrm{pr}_2)_{A}\big(\alpha\cdot\mathrm{pr}_1^{A}(-)\big)\colon A^*(\overline{X})\rightarrow A^*(\overline{Y})
$$
by~\cite[Theorem~4.14]{PShopf}. Let us compute the rank of 
$$
\mathrm{Hom}_{\,H^*_A(E\times E')-\mathrm{comod}}\left(A^*(\overline{X}),\,A^*(\overline{Y})\right).
$$

By Corollary~\ref{cl:Jinv} we have isomorphisms 
$$
\underline{\mathrm{pr}_1^A}\colon H^*_A(E)\xrightarrow{\cong} H^*_A(E\times E')\xleftarrow{\cong} H_A(E')\colon\underline{\mathrm{pr}_2^A}.
$$
In other words, the $H^*_A(E\times E')$-comodule structure on $A^*(\overline X)$ (resp., $A^*(\overline Y)$) coincides with the natural $H^*_A(E)$-comodule (resp., $H^*_A(E')$-comodule) structure on $A^*(G/B)$. 

By \cite[Lemma~5.5]{PShopf} $A^*(G/B)$ is a cofree $H_A^*(E)$- and $H_A^*(E')$-comodule of rank 
$$
r:=\frac{\rnk A^*(G/B)}{\rnk H_A^*(E)}=\frac{\rnk A^*(G/B)}{\rnk H_A^*(E')}.
$$ 
Let $H^*=H^*_A(E\times E')$, and $H^{\vee}=\mathrm{Hom}_{A^*(k)}\big(H^*,\,A^*(k)\big)$. Then $H^{\vee\vee}\cong H^*$ by Assumption~\ref{assum}, and therefore
\begin{equation}
\label{other-hand}
\rnk\mathrm{Hom}_{\,H^*-\mathrm{comod}}\left(A^*(\overline{X}),\,A^*(\overline{Y})\right)=\rnk\mathrm{End}_{\,H^\vee-\mathrm{mod}}((H^\vee)^r)=r^2\cdot\rnk H^\vee.
%
\end{equation}
However, 
$$
\overline A^*(X\times Y)\hookrightarrow\mathrm{Hom}_{\,H^*-\mathrm{comod}}\left(A^*(\overline{X}),\,A^*(\overline{Y})\right)
$$
by Lemma~\ref{linear-algebra}, and the quotient module
$$
{\mathrm{Hom}_{\,H^*-\mathrm{comod}}\left(A^*(\overline{X}),\,A^*(\overline{Y})\right)}\Big/{\overline A^*(X\times Y)}\Big.
$$
is trivial since it has to be $A^*(k)$-projective of rank $0$ by~\eqref{one-hand},~\eqref{other-hand} and Assumption~\ref{assum}.
In other words, every element in $A^*(G/B\times G/B)$ respecting the coaction of $H^*$ is $(E\times E')$-rational. 

In particular, since the diagonal $\Delta_{G/B}\in A^*(G/B\times G/B)$ corresponds to the identity morphism in $\mathrm{Hom}_{\,H^*-\mathrm{comod}}\left(A^*(G/B),\,A^*(G/B)\right)$, it is $(E\times E')$-rational, and by symmetry it is  $(E'\times E)$-rational as well. Then by the Rost nilpotence principle~\cite[Corollary~4.5]{GV} any preimage of $\Delta_{G/B}$ in $A^*(E/B\times E'/B)$ defines an isomorphism between $M_A(E/B)$ and $M_A(E'/B)$ (see~\cite[Proof of Lemma~2.1]{VY}).
\end{proof}

\subsection{$\mathrm K(1)$- and $\mathrm K(2)$-cases}

In the present section we denote by $G$ a split semi-simple algebraic group over $k$, and by $B$ its Borel subgroup. We will open this section with the following simple corollary of Theorem~\ref{tm:main} and~\cite[Proposition~7.10]{SeSe}.

\begin{prop}\label{prop:KnKn-1}
Let $E$ and $E'$ be two $G$-torsors over $\mathrm{Spec}(k)$. Assume that $\mathrm K(n+1)$-motives of $E/B$ and $E'/B$ are isomorphic. Then $\mathrm K(n)$-motives of $E/B$ and $E'/B$ are also isomorphic.
\end{prop}
\begin{proof}
Since $M_{\mathrm K(n+1)}(E/B)$ is isomorphic to $M_{\mathrm K(n+1)}(E'/B)$, the latter splits over $k(E/B)$. Then $M_{\mathrm K(n)}\big((E'/B)_{k(E/B)}\big)$ is also split by~\cite[Proposition~7.10]{SeSe}. By symmetry, $M_{\mathrm K(n)}\big((E/B)_{k(E'/B)}\big)$ is split as well, and therefore $\mathrm K(n)$-motives of $E/B$ and $E'/B$ are isomorphic by Theorem~\ref{tm:main}.
\end{proof}

Recall that for a $G$-torsor $E$ over $\mathrm{Spec}(k)$ we denote by $\,_pT(E)$ the $p$-torsion part of the subgroup of $\Br(k)$ generated by Tits algebras of $E$, see~\eqref{def-tits}.

\begin{tm}\label{tm:mainK1} Let $E$, $E'$ be two $G$-torsors over $\Spec(k)$. Then $M_{K(1)}(E/B)$ is isomorphic to $M_{K(1)}(E'/B)$ if and only if the $p$-primary components of the subgroups of Tits algebras of $E$ and $E'$ inside $\Br(k)$ coincide: $\,_pT(E)=\,_pT(E')$.
\end{tm}
\begin{proof}
    Assume that $M_{K(1)}(E/B)$ is isomorphic to $M_{K(1)}(E'/B)$. Denote the product of all Severi--Brauer varieties of Tits algebras of $E$ by $X$. Passing to the function field of $X$ we kill all Tits algebras, and by \cite[Theorem~4.2]{Pa} the subring of rational elements of $E/B\times E/B$ over $k(X)$ coincide with $K(1)(G/B\times G/B)$, hence $M_{K(1)}(E/B)_{k(X)}$ splits. Then $M_{K(1)}(E'/B)_{k(X)}$ splits as well, so the restriction map
    $$
    K_0((E'/B)_{k(X)})\otimes {\mathbb Z}_{(p)}\to K_0(G/B)\otimes {\mathbb Z}_{(p)}
    $$
    is an isomorphism, and again by \cite[Theorem~4.2]{Pa} the indices of all Tits algebras of $E'_{k(X)}$ are coprime to $p$. Hence the $p$-primary component of the subgroup of Tits algebras of $E'$ after passing to $k(X)$ becomes trivial. Now applying Lemma~\ref{tits-ker} we see that it is contained in the subgroup generated by all Tits algebras of $E$.

    Conversely, assume that the $p$-primary components of the subgroup of Tits algebras of $E$ and $E'$ inside $\Br(k)$ coincide. Passing to $k(E/B)$ we kill all Tits algebras of $E$. By the condition, passing further to an extension of degree coprime to $p$ we can kill all Tits algebras of $E'$ as well. Then the corestriction argument shows that $M_{K(1)}(E'/B)_{k(E/B)}$ splits, and it remains to apply Theorem~\ref{tm:main}.
\end{proof}

Assume that for a $G$-torsor $E$ over $\mathrm{Spec}(k)$ the $p$-primary component $\,_pT(E)$ of the subgroup of Tits algebras of $E$ is trivial. Recall that in this case we can define a $p$-primary component $\,_pR(E)$ of the subgroup of Rost invariants of $E$, see~\eqref{def-rost}. We will need the following simple generalization of~\cite[Theorem~9.1]{SeSe}.

\begin{lm}
\label{lm:sese}
Assume that for a $G$-torsor $E$ over $\mathrm{Spec}(k)$ the $p$-primary component $\,_pT(E)$ of the subgroup of Tits algebras of $E$ is trivial. Then $\,_pR(E)=0$ if and only if $M_{K(2)}(E/B)$ is split.
\end{lm}
\begin{proof}
There exist a finite field extension $L/k$ of degree prime to $p$ such that all Tits algebras of $E_L$ are trivial. Then $E_L$ comes from $E^{\mathrm{sc}}\in\mathrm H^1(L,\,G^{\mathrm{sc}})$ where $G^{\mathrm{sc}}$ denotes the simply-connected cover of $G_L$. Decompose $G^{\mathrm{sc}}=G_1\times\ldots\times G_r$ as a product if simple factors. Let $B^{\mathrm{sc}}$ denote the Borel subgroup of $G^{\mathrm{sc}}$ corresponding to $B_L$, and $B_i=B^{\mathrm{sc}}\cap G_i$ the Borel subgroups of $G_i$. Then $E^{\mathrm{sc}}=\prod_{i=1}^r E_i$ for $ E_i\in\mathrm H^1(L,\,G_i)$, and $E_L/B=E^{\mathrm{sc}}/B^{\mathrm{sc}}=\prod_{i=1}^r E_i/B_i$.
    
If $M_{K(2)}(E/B)$ is split, we conclude by Lemma~\ref{lm:pur-trans}\,iii) that $M_{K(2)}(E_i/B_i)$ are split $\forall\,i$. By~\cite[Theorem~9.1]{SeSe} this means that the $p$-components of the Rost invariants $\,_pr(E_i)$ are trivial $\forall\,i$, therefore $\,_pR(E)=\mathrm{cores}_{L/k}\left(\,_pR(E^{\mathrm{sc}})\right)=0$.

If $\,_pR(E)=0$, then $\,_pR(E^{\mathrm{sc}})=0$ as well (see Section~\ref{sec:rost}), in particular, $\,_pr(E_i)$ are trivial $\forall\,i$. By~\cite[Theorem~9.1]{SeSe} we conclude that $M_{K(2)}(E_i/B_i)$ are split $\forall\,i$, and therefore $M_{K(2)}(E_L/B)=\bigotimes_iM_{K(2)}(E_i/B_i)$ is also split. Since $[L:k]$ is prime to $p$, the claim follows.
\end{proof}

Now we are ready to prove the main result of the present paper.

\begin{tm}\label{tm:mainK2}
Let $E$, $E'$ be two $G$-torsors over $\Spec(k)$. Then $M_{K(2)}(E/B)$ is isomorphic to $M_{K(2)}(E'/B)$ if and only if the $p$-primary components of the subgroups of Tits algebras of $E$ and $E'$ inside $\Br(k)$ coincide: 
$
\,_pT(E)=\,_pT(E'),
$
 and
$$
\,_pR(E_{k(X)}) =\,_pR(E'_{k(X)})
$$
inside $\mathrm H^3\big(k(X),\,\mathbb Q/\mathbb Z(2)\big)$, where $X$ is the product of all Severi--Brauer varieties of Tits algebras in the $\,_pT(E)$.
\end{tm}
\begin{proof}
    Assume that $M_{K(2)}(E/B)$ is isomorphic to $M_{K(2)}(E'/B)$. Then by Proposition~\ref{prop:KnKn-1} $M_{K(1)}(E/B)$ is isomorphic to $M_{K(1)}(E'/B)$, and we can apply Theorem~\ref{tm:mainK1} to get $\,_pT(E)=\,_pT(E')$. 
    
Passing to the function field $k(X)(E/B)$ we see that $M_{K(2)}(E/B)_{k(X)(E/B)}$ splits, hence $M_{K(2)}(E'/B)_{k(X)(E/B)}$ splits. By Lemma~\ref{lm:sese} this means that the $p$-component of the subgroup of Rost invarians $\,_pR(E'_{k(X)(E/B)})$ is trivial. By Lemma~\ref{rost-ker} this implies that $\,_pR(E'_{k(X)})\subseteq\,_pR(E_{k(X)})$ and the result follows.

    Assume now that $\,_pT(E)=\,_pT(E')$, and $\,_pR(E_{k(X)})=\,_pR(E'_{k(X)})$. Since $(E/B)_{k(E/B)}$ is split, the Rost invariant subgroup $\,_pR(E_{k(E/B)})$ is trivial. Then $\,_pR(E'_{k(E/B)(X)})$ is also trivial by assumption, and therefore $M_{\mathrm K(2)}(E'/B)_{k(E/B)(X)}$ is split by Lemma~\ref{lm:sese}. However, all Tits algebras of $E$ are split over $k(E/B)$, therefore $X_{k(E/B)}$ is a product of projective spaces, in particular, the field $k(E/B)(X)$ is a purely transcendental extension of $k(E/B)$. Then by Lemma~\ref{lm:pur-trans}\,ii), $M_{\mathrm K(2)}(E'/B)_{k(E/B)}$ is also split. By symmetry, $M_{\mathrm K(2)}(E/B)_{k(E'/B)}$ is split as well, and therefore $M_{\mathrm K(2)}(E/B)$ is isomorphic to $M_{\mathrm K(2)}(E'/B)$ by Theorem~\ref{tm:main}. 

\end{proof}

\end{document}